\documentclass[a4paper,10pt]{article}
\usepackage{amssymb,latexsym}
\usepackage{amsmath,amsthm}
\usepackage{graphicx}
\usepackage{subfigure}
\usepackage{caption}
\usepackage{mathtools}
\usepackage{amsthm}
\usepackage{epstopdf}

\newcounter{case}
\newcounter{subcase}[case]
\newtheorem{theorem}{Theorem}[section]
\newtheorem{lemma}[theorem]{Lemma}

\newtheorem{conjecture}[theorem]{Conjecture}

\theoremstyle{remark}
\newtheorem{notation}[theorem]{Notation}
\newtheorem{remark}[theorem]{Remark}
\newtheorem{definition}[theorem]{Definition}

\newtheorem{observation}[theorem]{Observation}

\newtheorem{problem}{Problem}

\renewcommand{\ge}{\geqslant}
\renewcommand{\le}{\leqslant}

\def\cref#1{Corollary~$\ref{#1}$}

\newcommand{\A}{\mathcal{A}}
\newcommand{\B}{\mathcal{B}}

\newcommand{\ip}{\mathcal{P}}

\begin{document}
\date{}

\title{Distance magic labeling in complete 4-partite graphs}

\author{
Daniel Kotlar\\
\small Department of Computer Science, Tel-Hai College, Upper Galilee, Israel\\
}

\maketitle

\begin{abstract}
Let $G$ be a complete $k$-partite simple undirected graph with parts of sizes $p_1\le p_2\cdots\le p_k$. Let $P_j=\sum_{i=1}^jp_i$ for $j=1,\ldots,k$. It is conjectured that $G$ has distance magic labeling if and only if $\sum_{i=1}^{P_j} (n-i+1)\ge j{{n+1}\choose{2}}/k$ for all $j=1,\ldots,k$. The conjecture is proved for $k=4$, extending earlier results for $k=2,3$.
\end{abstract}


\section{Introduction}
Let $G=(V,E)$ be a finite, simple, undirected graph of order $n$. Denoting $[n]=\{1,2,\ldots,n\}$, a \emph{distance magic labeling} of $G$ \cite{miller2003} (or \emph{sigma labeling} \cite{beena2009}) is a bijection $f:V\rightarrow[n]$ such that for all $x\in V$,
\begin{equation}\label{eq1}
\sum_{y\in N(x)}f(y)=\mathbf{c}
\end{equation}
for a constant $\mathbf{c}$, independent of $x$ ($N(x)$ is the set of vertices adjacent to $x$).

Miller, Rodger and Simanjuntak \cite{miller2003} showed that if $G$ is the $2kr$-regular multipartite graph $H\times \overline{K_{2k}}$ then $G$ has a distance magic labeling ($\overline{K_{2k}}$ is the complement of the graph $K_{2k}$). They also showed that for the complete symmetric multipartite graph $H_{n,p}$ with $p$ parts and $n$ vertices in each part, $H_{n,p}$ has a distance magic labeling if and only if either $n$ is even or both $n$ and $p$ are odd. In addition, they gave necessary and sufficient conditions for a complete multipartite graph $K_{p_1,p_2,\ldots,p_k}$ (where the parts are not necessarily of equal sizes) to have a distance magic labeling for $k=2,3$. The result for $k=2$ also appears in \cite{beena2009}. For more results and surveys on distance magic labeling see \cite{anholcer2014distance, arumugam2012distance, gallian2009dynamic}.

It has been observed that the problem of characterizing the complete multipartite graphs which have a distance magic labeling is equivalent to a problem on partitions of $[n]$: let $V=V_1\cup V_2\cup\cdots\cup V_k$ be the parts of $G=K_{p_1,p_2,\ldots,p_k}$ with sizes $p_1,p_2,\ldots,p_k$, respectively, so that $p_1+p_2+\cdots+p_k=n$. Then, $G$ has a distance magic labeling if and only if there exists a bijection $f:V\rightarrow[n]$ such that for all $j=1,\dots,k$, $\sum_{i=1,i\ne j}\sum_{x\in V_i}f(x)=\mathbf{c}$, where $\mathbf{c}$ is a constant. This is equivalent to $\mathbf{c}={{n+1}\choose{2}}-\sum_{x\in V_j}f(x)$ for all $j=1,\dots,k$, or $\sum_{x\in V_j}f(x)={{n+1}\choose{2}}/k$ for all $j=1,\dots,k$. Denote $s^{n,k}={{n+1}\choose{2}}/k$. The problem can be reformulated as follows:

\begin{problem}\label{prob1}
Let $n,k$ and $p_1,\dots,p_k$ be positive integers such that $p_1+\cdots+p_k=n$ and $s^{n,k}$ is an integer. When is it possible to find a partition of the set $[n]$ into $k$ subsets of sizes $p_1,\dots,p_k$, respectively, such that the sum of the elements in each subset is $s^{n,k}$?
\end{problem}

Anholcer, Cichacz and Peterin \cite{An2015} related this problem to a different problem in vertex labeling: let $G$ be the graph obtained from the cycle $C_k$ by replacing every vertex $v_i$ by a clique $K[v_i]$ of some order $p_i$, and joining all the vertices of each $K[v_i]$ with all the vertices of $K[v_j]$ whenever $v_j$ is a neighbor of $v_i$ in $C_k$. Then, consider the problem of finding a \emph{closed distance magic labeling} of $G$, that is, a magic labeling where the sum in (\ref{eq1}) includes $f(x)$. Clearly, if the partition in Problem~\ref{prob1} is solved for $k$, then $G$ has a closed distance magic labeling. A necessary condition for such a partition to exist was observed in \cite{An2015}:

\begin{observation}
Assume the setup of Problem~\ref{prob1} with $p_1,\ldots,p_k$ given in a non-decreasing order. Let $P_j=\sum_{i=1}^jp_i$ for $j=1,\ldots,k$. If the mentioned partition of $[n]$ exists, then
\begin{equation}\label{nec_cond}
\sum_{i=1}^{P_j} (n-i+1)\ge js^{n,k} \quad {\textrm {for all}}\quad j=1,\ldots,k.
\end{equation}
\end{observation}

It is conjectured here (Conjecture~\ref{conj1}) that this condition is sufficient. Before formulating the conjecture we shall need some notation and definitions:

\begin{definition}
Let $n$ be a positive integer and let $\ip=\{p_1, p_2,\cdots, p_k\}$ be a set of positive integers satisfying $\sum_{i=1}^k p_i=n$. We say that a partition $\A=\{A_1,A_2,\dots,A_k\}$ of $[n]$ \emph{implements} the sequence $\{p_i\}_{i=1}^k$, if there is a bijection $f:\A\rightarrow\ip$ such that $|A_i|=f(A_i)$ for all $i=1,\ldots,k$.
\end{definition}

\begin{definition}
For any set of integers $A$, the \emph{sum} of $A$, denoted $S(A)$, is the sum of the elements in $A$.
\end{definition}

\begin{definition}
Let $n$ and $k$ be positive integer such that $s^{n,k}$ is an integer. We say that the partition $\A=\{A_1,A_2,\dots,A_k\}$ of $[n]$ is \emph{equitable} if $S(A_i)=s^{n,k}$ for all $i=1,\ldots,k$.
\end{definition}

\begin{conjecture}\label{conj1}
Let $k<n$ be positive integers such that $s^{n,k}$ is an integer. Let $1<p_1\le p_2\le\cdots\le p_k$ be positive integers such that $\sum_{i=1}^k p_i=n$. There exists an equitable partition of $[n]$ implementing $\{p_i\}_{i=1}^k$ if and only if (\ref{nec_cond}) holds.
\end{conjecture}

This conjecture was in fact proved in \cite{miller2003} for $k=2,3$. The case $k=2$ was independently proved in \cite{beena2009}. The main result in this paper is:

\begin{theorem}\label{thm1}
Conjecture \ref{conj1} holds for $k=4$.
\end{theorem}

The approach used for proving Theorem~\ref{thm1} also provides simple proofs for the cases $k=2,3$ (Remarks~\ref{rem:1} and \ref{rem:2}).

Thus, the $k$-partite complete graphs for which a magic distance labeling exits are characterized for $k\le 4$. This also solves the above mentioned closed distance magic labeling problem from \cite{An2015} for $k=4$.

Note that the case where $p_1=1$ is left out in Conjecture~\ref{conj1}, as this case is different and straightforward. Suppose $p_1=1$. In order for a partition which implements $\{p_i\}_{i=1}^k$ to be equitable, the set of size 1 must be $\{n\}$, and thus, $n=s^{n,k}$, which is equivalent to $k=(n+1)/2$. Since there can be only one set of size 1 in an equitable partition, all the other sets must be of size 2. Thus, we can partition the remaining $[n-1]$ into $k-1$ pairs, in the obvious way, to get an equitable partition.

\section{Preliminary results and notation}
We first introduce some notation and definitions and prove some preliminary results for general $k$. We assume that $k$ and $n$ are such that $s^{n,k}$ is an integer.

\begin{notation}
For a partition $\A=\{A_1,A_2,\dots,A_k\}$ of $[n]$, we denote $d(\A)=\sum_{i=1}^k (S(A_i)-s^{n,k})^2$.
\end{notation}

For any set $A$ and an element $x$ we use the notation $A-x$ for $A\setminus\{x\}$ and $A\cup x$ for $A\cup\{x\}$.

\begin{notation}
Let $\A=\{A_1,A_2,\dots,A_k\}$ be a partition of $[n]$ and suppose $a\in A_i$ and $b\in A_j$, where $a<b$ and $i\ne j$. We shall denote by $\chi_{a,b}$ the operator that acts on $\A$ by exchanging $a$ and $b$ between $A_i$ and $A_j$. The result is a new partition $\chi_{a,b}(\A)$, which we shall denote $\A_{a,b}$, where $A_i$ is replaced by $A_i-a\cup b$ and $A_j$ is replaced by $A_j-b\cup a$.
\end{notation}

\begin{observation}\label{obs:1}
Let $\A=\{A_1,A_2,\dots,A_k\}$ be partition of $[n]$. Suppose $a\in A_i$ and $b\in A_j$ where $i\ne j$. Then, $d(\A_{a,b})=d(\A)$ if $b-a=S(A_j)-S(A_i)$, $d(\A_{a,b})>d(\A)$ if $b-a > S(A_j)-S(A_i)$, and $d(\A_{a,b})<d(\A)$ if $b-a < S(A_j)-S(A_i)$.
\end{observation}
\begin{proof}
Let $t=b-a$ and $u=S(A_j)-S(A_i)$. We have $d(\A_{a,b})=d(\A)+2t(t-u)$.
\end{proof}

\begin{definition}
For a given sequence $\ip=\{p_i\}_{i=1}^k$ satisfying $\sum_{i=1}^k p_i=n$, let $d^{\ip}_{\textrm{min}}$ be the minimal value of $d(\A)$ over all partitions of $[n]$ which implement $\ip$.
A partition $\A$ implementing  $\ip$, such that $d(\A)=d^{\ip}_{\textrm{min}}$ will be called a \emph{minimal partition}.
\end{definition}

\begin{definition}
Let $\A$ be a partition of $[n]$ and let $A\in\A$. We say that $A$ is \emph{low} if $S(A)<s^{n,k}$, $A$ is \emph{high} if $S(A)>s^{n,k}$, and $A$ is \emph{exact} if $S(A)=s^{n,k}$.
\end{definition}

\begin{observation}\label{obs:2}
Suppose there is no equitable partition of $[n]$ implementing $\{p_i\}_{i=1}^k$ and let $\A$ be a minimal partition. Then there is no $c\in [n]$ such that $c$ is in a low set of $\A$ and $c+1$ is in a high set.
\end{observation}
\begin{proof}
Assume the contrary, so that $c\in A_i$ and $c+1\in A_j$ where $A_i$ is a low set and $A_j$ is a high set. We have $S(A_j)-S(A_i)>1$. So, $d(\A_{c,c+1})<d(\A)$, by Observation~\ref{obs:1}, contradicting the minimality of $\A$.
\end{proof}

\begin{lemma}\label{lemma:1}
Let $\ip=\{p_i\}_{i=1}^k$ be such that $\sum_{i=1}^kp_i=n$ and satisfies (\ref{nec_cond}). Suppose there is no equitable partition implementing $\ip$ and let $\A$ be a minimal partition of $[n]$ implementing $\ip$. Then, there exists a number $t$ in a low set of $\A$ such that $t+1$ is in an exact set of $\A$ and there exists a number $s$ in an exact set of $\A$ such that $s+1$ in a high set of $\A$.
\end{lemma}
\begin{proof}
Let $\A^l$ and  $A^e$ be the collections of low and exact sets of $\A$, respectively. Let $l=|\A^l|$ and $e=|A^e|$. Let $L=\sum_{A\in\A^l}|A|$ and $E=\sum_{A\in\A^e}|A|$.
If all the elements in the low sets of $\A$ are greater than all the elements in the other sets, then  $\bigcup_{A\in\A^l}A=\{n-L+1,\dots,n-1,n\}$, and we have

\begin{equation*}
    \sum_{i=1}^{P_l} (n-i+1) = \sum_{i=1}^{p_1+\cdots +p_l} (n-i+1) < l\cdot s^{n,k},
\end{equation*}

contradicting (\ref{nec_cond}). So there must be an element $t$ in a low set such that $t+1$ is not in a low set. By Observation~\ref{obs:2}, $t+1$ must be in an exact set.
Now, suppose, for contradiction, that all the elements in the low and exact sets of $\A$ are greater than all the elements in the high sets of $\A$. We have

\begin{equation*}
    \sum_{i=1}^{P_{l+e}} (n-i+1) = \sum_{i=1}^{p_1+\cdots +p_{l+e}} (n-i+1) < (l+e) s^{n,k},
\end{equation*}

contradicting (\ref{nec_cond}). So, there must be $s$ in a low or exact set such that $s+1$ is in a high set. By Observation~\ref{obs:2}, $s$ must be in an exact set.
\end{proof}

\begin{remark}\label{rem:1}
It follows from Lemma~\ref{lemma:1} that if $\ip$ satisfies (\ref{nec_cond}) and there is no equitable partition implementing $\ip$, then $k\ge 3$. This implies Conjecture~\ref{conj1} for $k=2$.
\end{remark}

\begin{lemma}\label{lemma:2}
Let $\ip=\{p_i\}_{i=1}^k$ be such that $\sum_{i=1}^kp_i=n$ and satisfies (\ref{nec_cond}). If there is no equitable partition implementing $\ip$, then every minimal partition implementing $\ip$ contains at least one exact set, one low set with sum $s^{n,k}-1$, and one high with sum $s^{n,k}+1$.
\end{lemma}
\begin{proof}
Let $\A$ be a minimal partition implementing $\ip$. By Lemma~\ref{lemma:1}, there must be a low set $X$, an exact set $Z$ and a number $a\in X$ such that $a+1\in Z$ and an exact set $W$, a high set $Y$, and a number $b\in W$ such that $b+1\in Y$. Since $\A$ is minimal, we must have $S(X)=s^{n,k}-1$ and $S(Y)=s^{n,k}+1$, by Observation~\ref{obs:1}.
\end{proof}

In the discussion ahead we shall use diagrams with two horizontal lines. High sets with sum $s^{n,k}+1$ appear above the lines, exact sets appear between the lines, and low sets with sum $s^{n,k}-1$ are drawn below the lines. An arrow points from a number to its successor. For example, Figure~\ref{fig1} illustrates the setup in the proof of Lemma~\ref{lemma:2}.

\begin{figure}[h!]
\begin{center}
\includegraphics[scale=0.35]{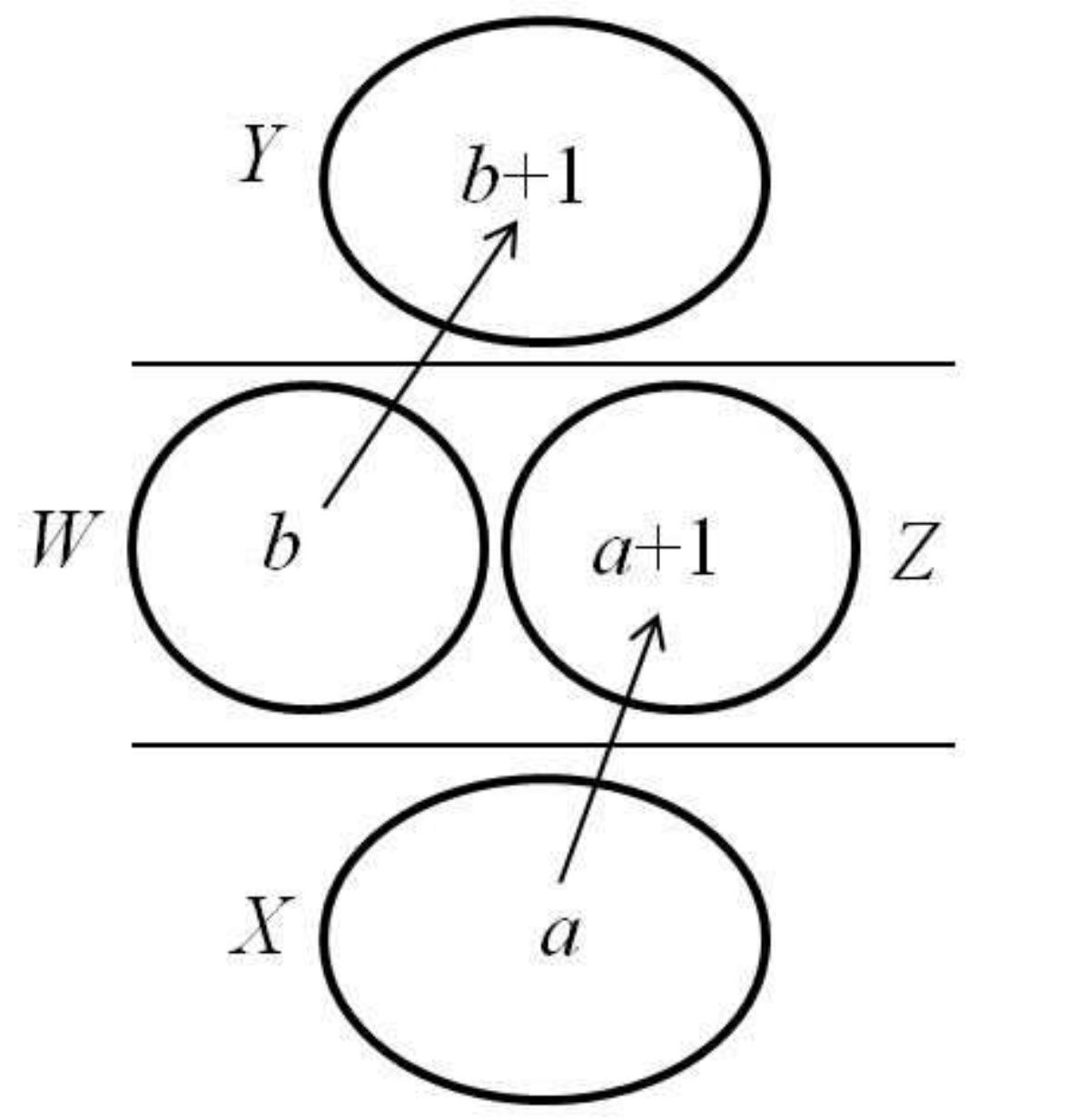}
\end{center}
\caption{}
\label{fig1}
\end{figure}

\begin{definition}
We say that two partitions $\A$ and $\A'$ of $[n]$ are \emph{equivalent}, denoted $\A\equiv \A'$, if there exists a bijection $f:\A\rightarrow\A'$ such that $S(f(A))=S(A)$ for all $A\in \A$.
\end{definition}

Clearly, if $\A\equiv \A'$, then $d(\A)= d(\A')$.

\begin{lemma}\label{lemma:3}
If $\A$ is a minimal partition then the configurations illustrated in Figure~\ref{fig2} are not possible, given that $b\ne a+1$ in all figures, $c\ne b+1$ in Figure~\ref{fig2b}, $x\ne b+1$ in Figures~\ref{fig2c} and \ref{fig2e}, and $x+2\ne a$ in Figures~\ref{fig2d} and \ref{fig2e}.
\end{lemma}

 \begin{figure}[h!]
  \centering
  \subfigure[]{\label{fig2a}\includegraphics[scale=0.25]{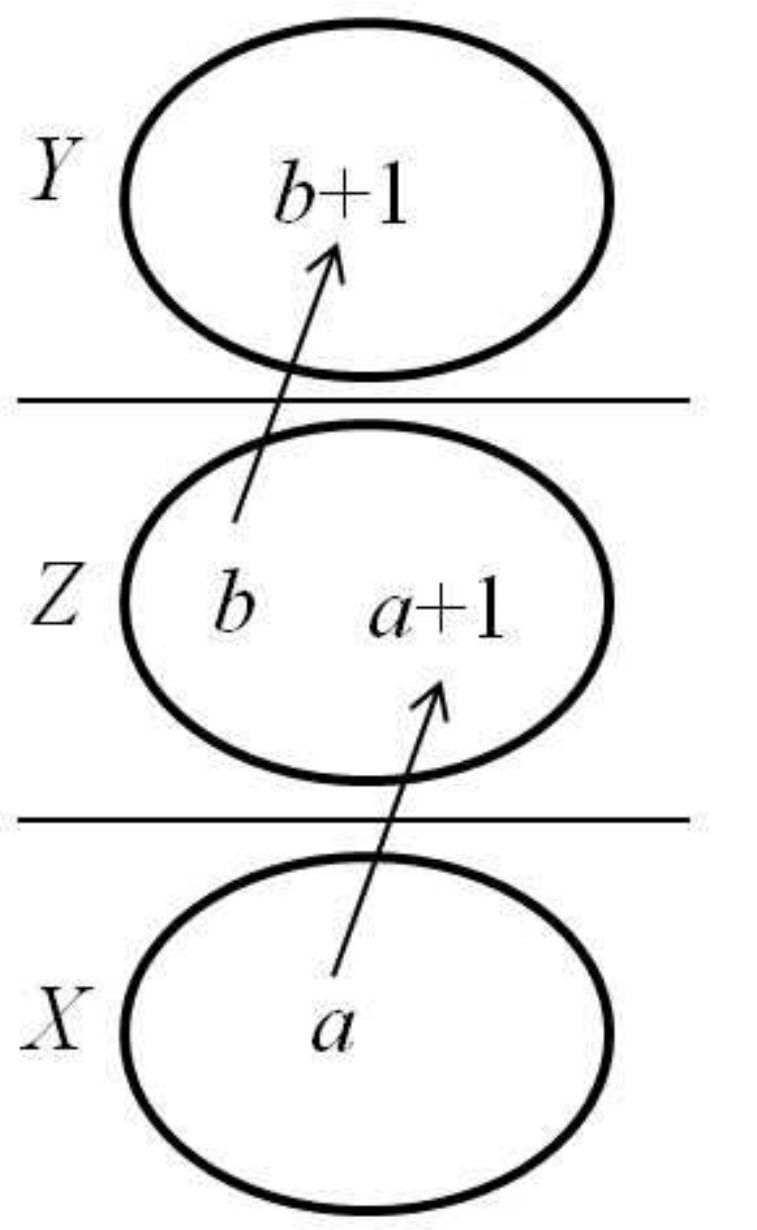}}
  \subfigure[]{\label{fig2b}\includegraphics[scale=0.25]{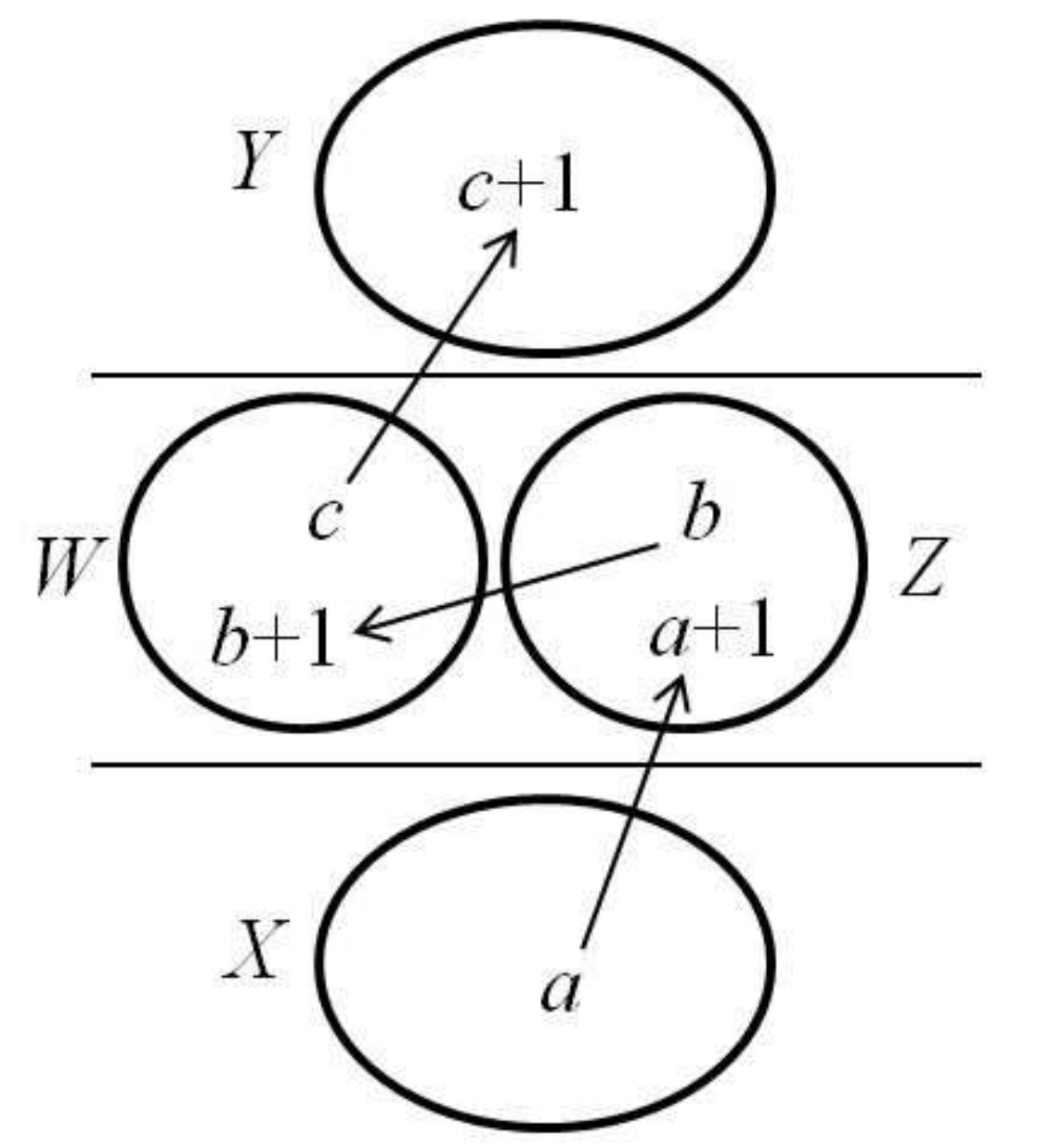}}
  \subfigure[]{\label{fig2c}\includegraphics[scale=0.25]{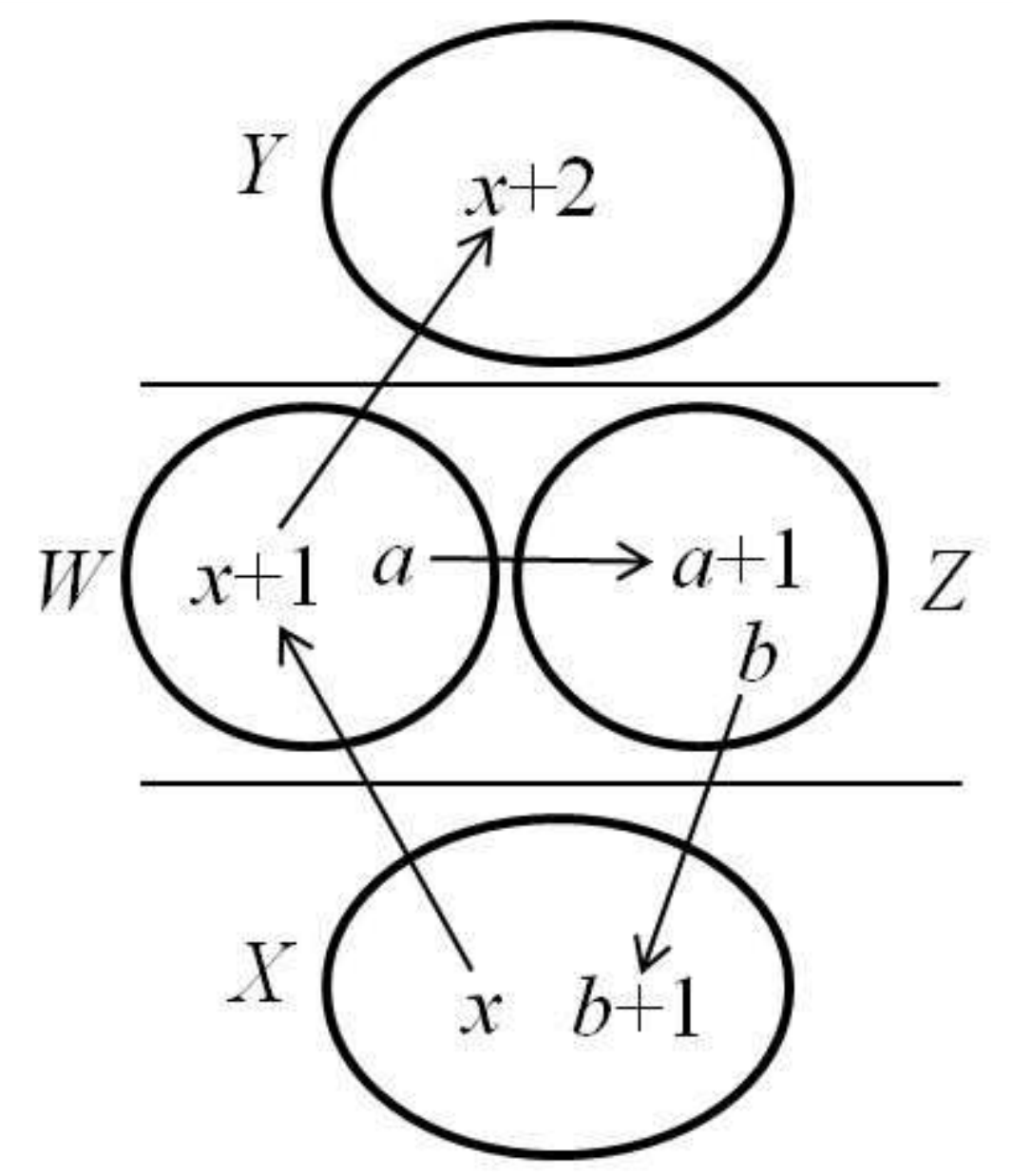}}
  \subfigure[]{\label{fig2d}\includegraphics[scale=0.25]{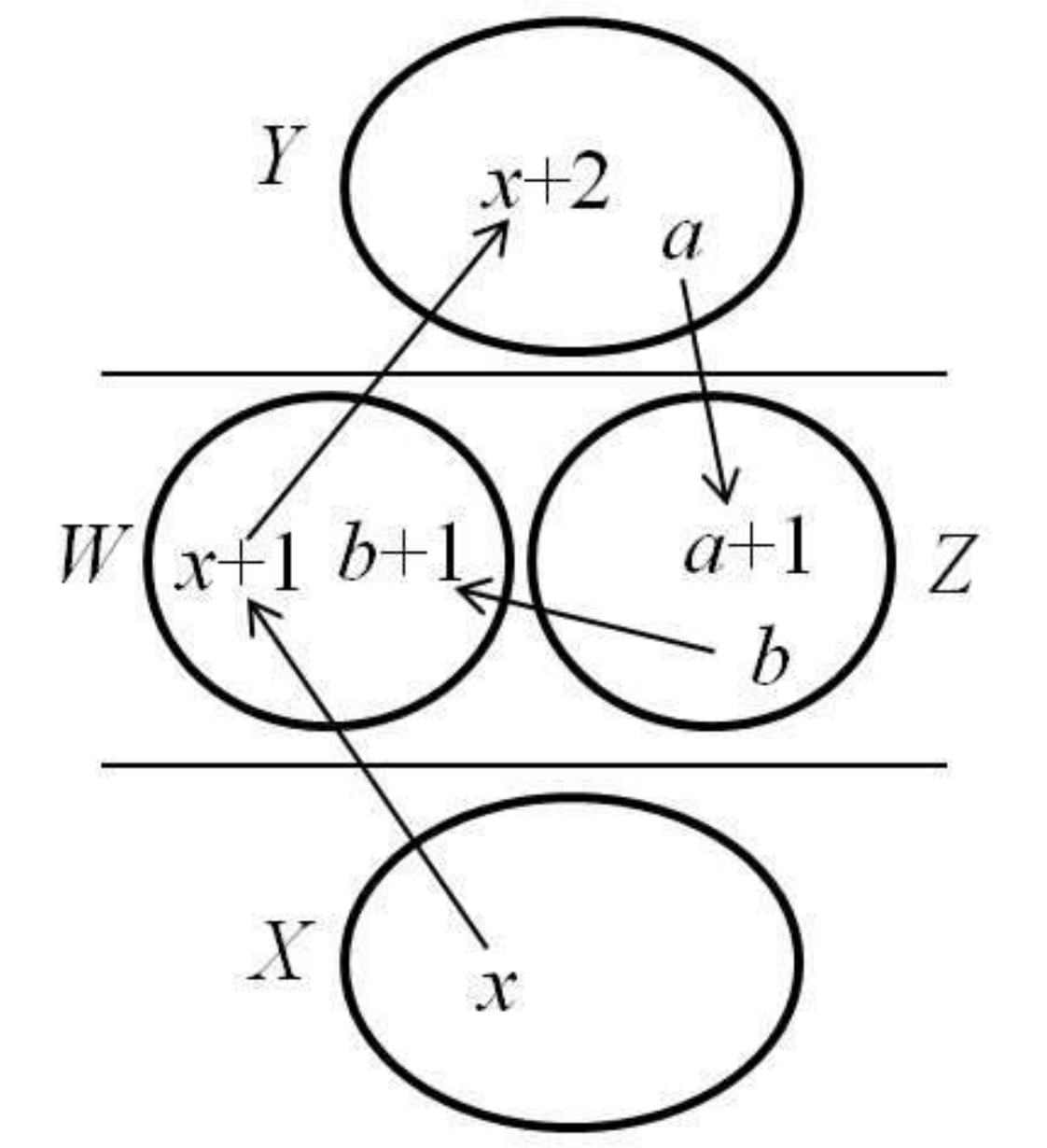}}
  \subfigure[]{\label{fig2e}\includegraphics[scale=0.25]{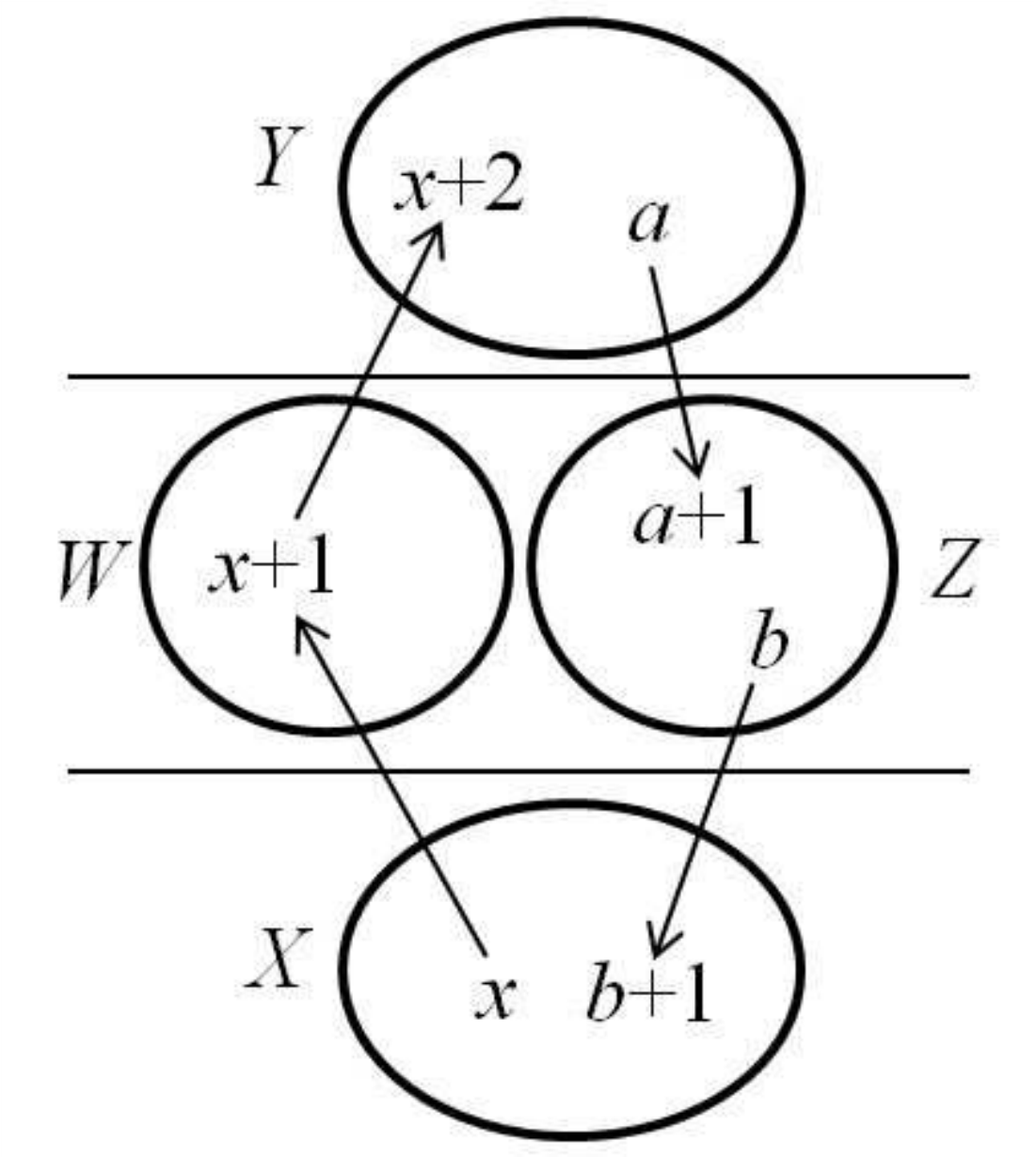}}
  \caption{}
  \label{fig2}
\end{figure}

\begin{proof}
Suppose the configuration in Figure~\ref{fig2a} exists. Applying $\chi_{a,a+1}$ we obtain an equivalent partition with $b$ in a low set and $b+1$ in a high set, contradicting Observation~\ref{obs:2}. Suppose the configuration in Figure~\ref{fig2b} exists. Applying $\chi_{c,c+1}$, the element $b+1$ gets pushed up to the high set and we obtain a configuration as in Figure~\ref{fig2a}. Suppose the configuration in Figure~\ref{fig2c} or Figure~\ref{fig2d} exists. Applying $\chi_{x,x+2}$, we obtain an equivalent partition with a configuration as in Figure~\ref{fig2b}. If the Configuration in Figure~\ref{fig2e} exits, then applying $\chi_{x,x+2}$ yields an equivalent partition with a configuration as in Figure~\ref{fig2a}.
\end{proof}

In general, if $A,B\in \A$, $S(B)=S(A)+1$, $a\in A$ and $a+1\in B$, then, clearly $\A_{a,a+1}\equiv\A$. Such actions of the form $\chi_{a,a+1}$ will be common in our discussion and we won't always mention the obvious fact that the partitions are equivalent.

\begin{remark}\label{rem:2}
Let $\ip=\{p_1,p_2,p_3\}$ be such that $p_1+p_2+p_3=n$ and satisfies (\ref{nec_cond}) for $k=3$. Suppose there is no equitable partition implementing $\ip$ and let $\A$ be a minimal partition implementing $\ip$. By Observation~\ref{obs:2}, Lemma~\ref{lemma:2} and Lemma~\ref{lemma:3}(a) we must have the setup illustrated in Figure~\ref{fig3}. Note that we may assume that $n\in A_3$ (if $n\in A_2$ we apply $\chi_{t+1,t+2}$ and if $n\in A_1$ we apply $\chi_{t+1,t+2}\circ\chi_{t,t+1}$). Let $s$ be the maximal element in $A_2-(t+1)$. Since $s\ne n$ and $n\not\in A_1$ we must have $s+1\in A_3$, by Observation~\ref{obs:2}. But this yields a configuration as in Lemma~\ref{lemma:3}(a). Thus, Conjecture~\ref{conj1} holds for $k=3$.
\end{remark}

\begin{figure}[h!]
\begin{center}
\includegraphics[scale=0.35]{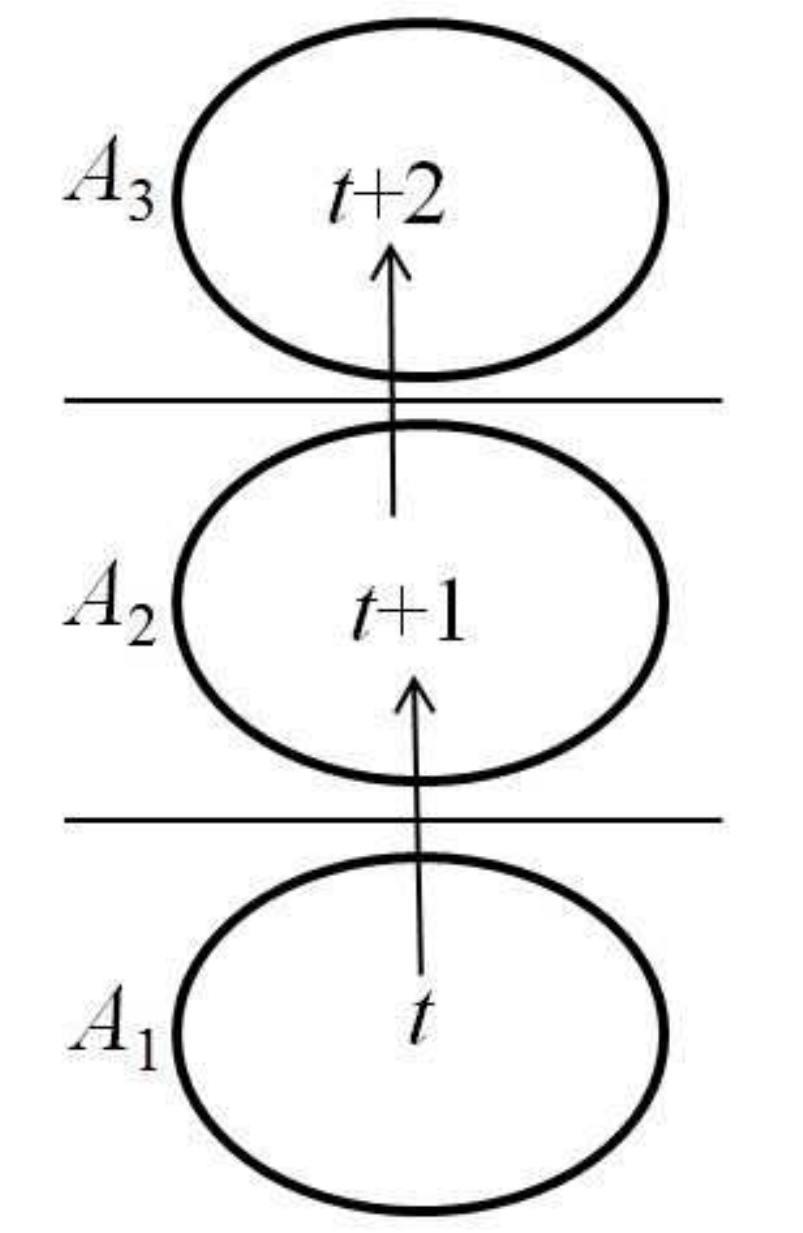}
\end{center}
\caption{}
\label{fig3}
\end{figure}

\section{Some lemmas for the case $k=4$}

\begin{lemma}\label{lemma:4}
Let $\ip=\{p_i\}_{i=1}^4$ be such that $\sum_{i=1}^4p_i=n$ and satisfies (\ref{nec_cond}). If there is no equitable partition of $[n]$ implementing $\ip$, then every minimal partition implementing $\ip$ has one low set with sum $s^{n,4}-1$, one high set with sum $s^{n,4}+1$, and two exact sets.
\end{lemma}

\begin{proof}
By Lemma~\ref{lemma:2} there is one low set and one high set with the indicated properties, and one exact set. Since the sum of all the elements is $4s^{n,4}$, the fourth set must be exact.
\end{proof}

\begin{definition}
Let $\A$ be a partition of $[n]$ implementing the sequence $\{p_i\}_{i=1}^k$ and assume $\A$ is not equitable. We define the \emph{width} of $\A$, denoted $\omega(\A)$, as the minimal value of $y-x$ over all $x,y\in [n]$ such that $y>x$ and such that $y$ is in a high set of $\A$ and $x$ is in a low set. If there are no such $x$ and $y$ we define $\omega(\A)=\infty$.
\end{definition}

\begin{lemma}\label{lemma:5}
Let $\ip=\{p_i\}_{i=1}^4$ be such that $\sum_{i=1}^4p_i=n$ and satisfies (\ref{nec_cond}) for $k=4$. Suppose there is no equitable partition of $[n]$ implementing $\ip$. Then, there exists a minimal partition of $[n]$ with finite width.
\end{lemma}

\begin{proof}
Let $\A=\{A_1,\dots ,A_4\}$ be a minimal partition of $[n]$ implementing $\{p_i\}_{i=1}^4$ and assume $\omega(\A)=\infty$. We shall show that there is an equivalent partition with finite width.
By Lemma~\ref{lemma:4}, we may assume that $S(A_1)= s^{n,4}-1$, $S(A_2)= S(A_3)=s^{n,4}$ and $S(A_4)= s^{n,4}+1$. By Lemma~\ref{lemma:1} there exist $a\in A_1$ such that $a+1\in A_2\cup A_3$, and $b\in A_2\cup A_3$ such that $b+1\in A_4$. We may assume that $a+1\in A_2$. If $b\in A_2$, then, by Lemma~\ref{lemma:3}(a), we must have $b= a+1$. But then $\omega(\A)\le 2$, contradicting our assumption. So, we must have $b\in A_3$. If there exist $z<w$ such that $z\in A_2$ and $w\in A_4$, then $z\ne a+1$, since $\omega(\A)=\infty$. Applying $\chi_{a,a+1}$ will yield an equivalent partition with finite width, since it has $z$ in a low set and $w$ in a high set. If there are $z<w$ such that $z\in A_1$ and $w\in A_3$, then $w\ne b$, since $\omega(\A)=\infty$, and applying $\chi_{b,b+1}$ will yield another equivalent partition with finite width. Thus, there must exist $z<w$ such that $z\in A_2$ and $w\in A_3$ (otherwise, all the elements of $A_1\cup A_2$ are greater than all the elements in $A_3\cup A_4$, contradicting (\ref{nec_cond})).  If $w= b$, then $z<b+1\in A_4$ and we have already considered such a setting. So we assume $w\ne b$. In this case we can apply $\chi{a,a+1}\circ\chi_{b,b+1}$ and obtain an equivalent partition with $w$ in the high set and $z$ in the low set, and thus, of finite width.
\end{proof}

\begin{lemma}\label{lemma:6}
Let $\A=\{X,Y,Z,W\}$ be a minimal partition of $[n]$ implementing $\{p_i\}_{i=1}^4$, of minimal width among all such minimal partitions. Assume that $S(X)=s^{n,4}-1, S(Y)=s^{n,4}+1$ and $S(Z)=S(W)=s^{n,4}$. Let $x\in X$ and $y\in Y$ satisfy $y-x=\omega(\A)$. Suppose $x+1\in A$ and $y-1\in B$, where $A,B\in \{Z,W\}$. Then
\begin{enumerate}
  \item [(i)] $y-x=2$ or 3, and $y-x=2$ if and only if $A=B$.
  \item [(ii)] $A-(x+1)$ contains no element $a$ such that $a-1\in X$
  \item [(iii)] $A-(x+1)$ contains no element $a$ such that $a+1\in (X-x)$.
  \item [(iv)] $B-(y-1)$ contains no element $b$ such that $b+1\in Y$
  \item [(v)] $B-(y-1)$ contains no element $b$ such that $b-1\in (Y-y)$.
  \item [(vi)] If $A=B$, then $Y-y$ contains no element $c$ such that $c+1\in X-x$.
\end{enumerate}
Figure~\ref{fig4} indicates in dotted lines the illegal configurations of (ii)-(vi) for the two cases implied by (i). Note that in Figure~\ref{fig4b} the cases (ii) and (iv) are already known from Lemma~\ref{lemma:3}(a).
\end{lemma}

 \begin{figure}[h!]
  \centering
  \subfigure[]{\label{fig4a}\includegraphics[scale=0.35]{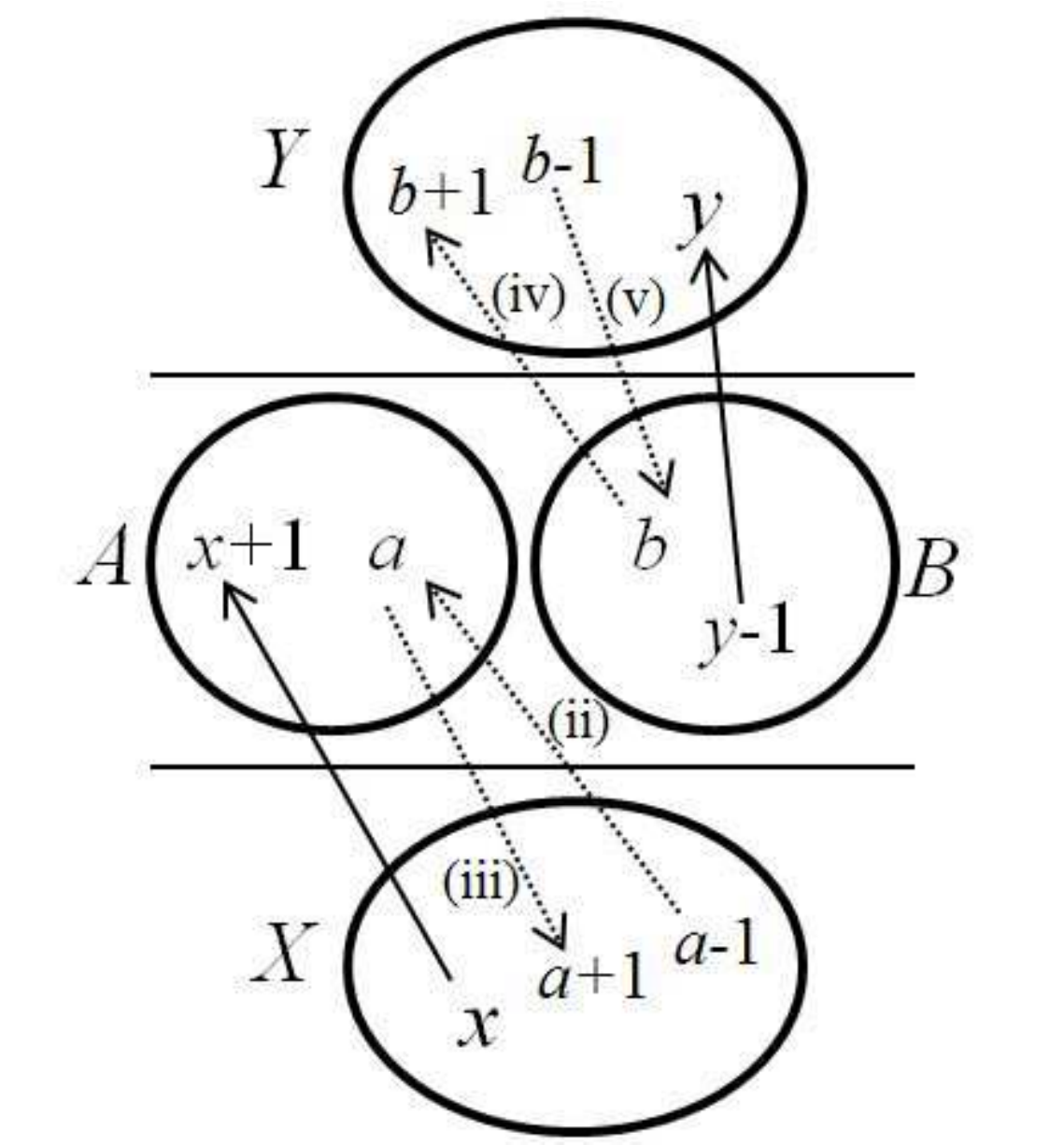}}
  \subfigure[]{\label{fig4b}\includegraphics[scale=0.35]{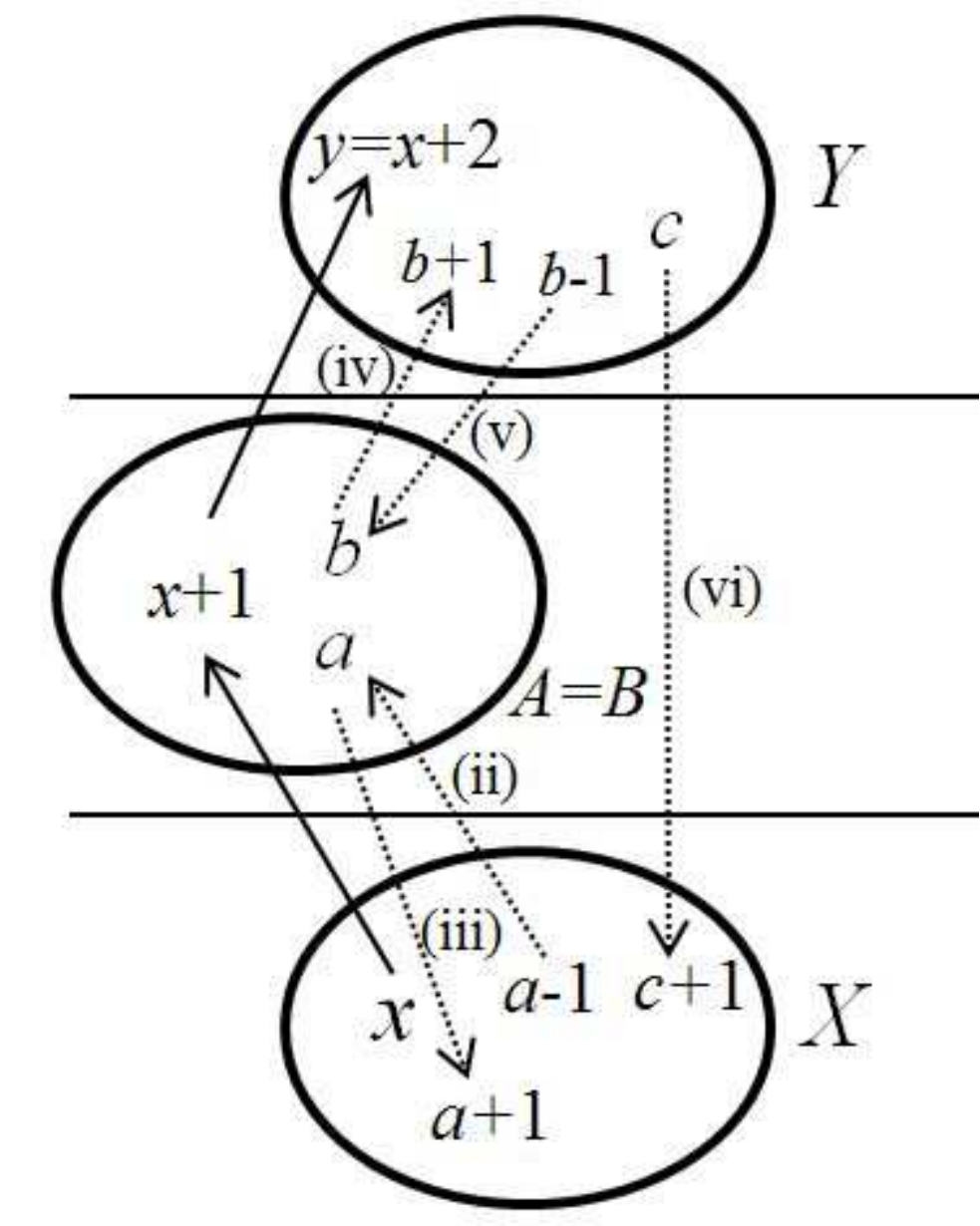}}
  \caption{}
  \label{fig4}
\end{figure}

\begin{proof}
First note that $\A$ exists by Lemma~\ref{lemma:5}. (i) The existence of $z\in A$ satisfying $x+1<z<y$ is not possible since $z$ would be in the low set of $\A_{x,x+1}$ while $y$ is in the high set, contradicting the minimality of $\omega(\A)$. Similarly, the existence of $w\in B$ such that $x<w<y-1$ would imply that $w$ is in the high set of $\A_{y-1,y}$ and $x$ in the low set. Again, a contradiction. Thus, $y-x\le 3$. If $x+1$ and $y-1$ are in the same set, then they are equal, by Lemma~\ref{lemma:3}(a), and in this case $y-x=2$.

Now, let $a\in A-(x+1)$.
(ii) If $a-1\in X$, then $\A_{a,a-1}$ has $x+1$ in a low set and $y$ in a high set, contradicting the minimality of $\omega(\A)$.
(iii) If $a+1\in X-x$ we apply $\chi_{x,x+1}$ and obtain a setup as in (ii). The proofs of (iv) and (v) are similar.

(vi) Suppose such $c\in Y-y$ exists. Note that $\A_{x,x+2}\equiv\A$, but now $c$ is in the low set and $c+1$ is in the high set, contradicting Observation~\ref{obs:2}.
\end{proof}

\begin{lemma}\label{lemma:7}
Let $\A=\{A_i\}_{i=1}^4$ be a minimal partition of $[n]$ implementing $\{p_i\}_{i=1}^4$, of minimal width among all such minimal partitions. Then, the configurations illustrated in Figure~\ref{fig5} are not possible, assuming $x+2\ne d$ (Figures~\ref{fig5d} and \ref{fig5f}) and $x\ne d+2$ (Figures~\ref{fig5e} and \ref{fig5f}).
\end{lemma}

 \begin{figure}[h!]
  \centering
  \subfigure[]{\label{fig5a}\includegraphics[scale=0.23]{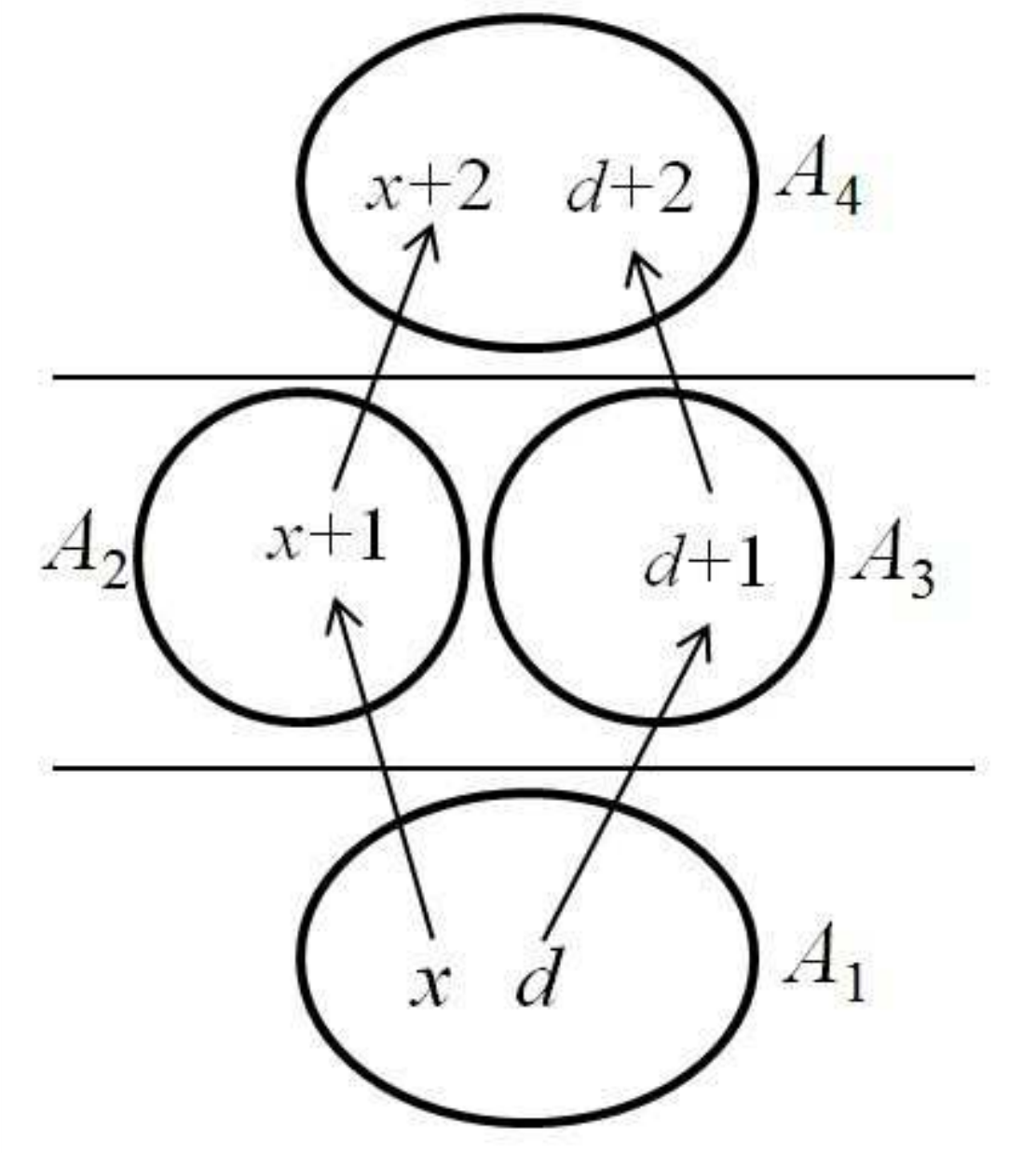}}
  \subfigure[]{\label{fig5b}\includegraphics[scale=0.23]{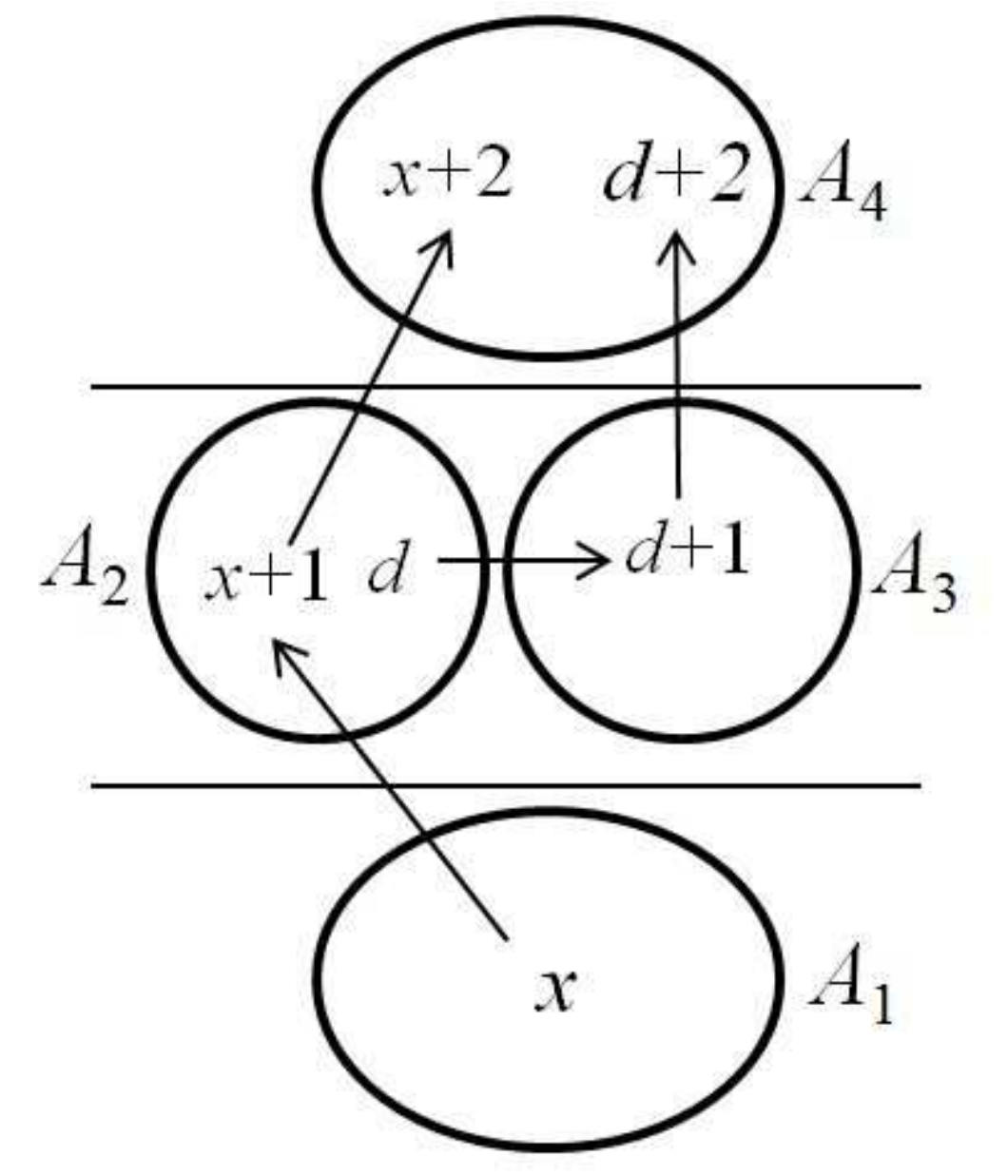}}
  \subfigure[]{\label{fig5c}\includegraphics[scale=0.23]{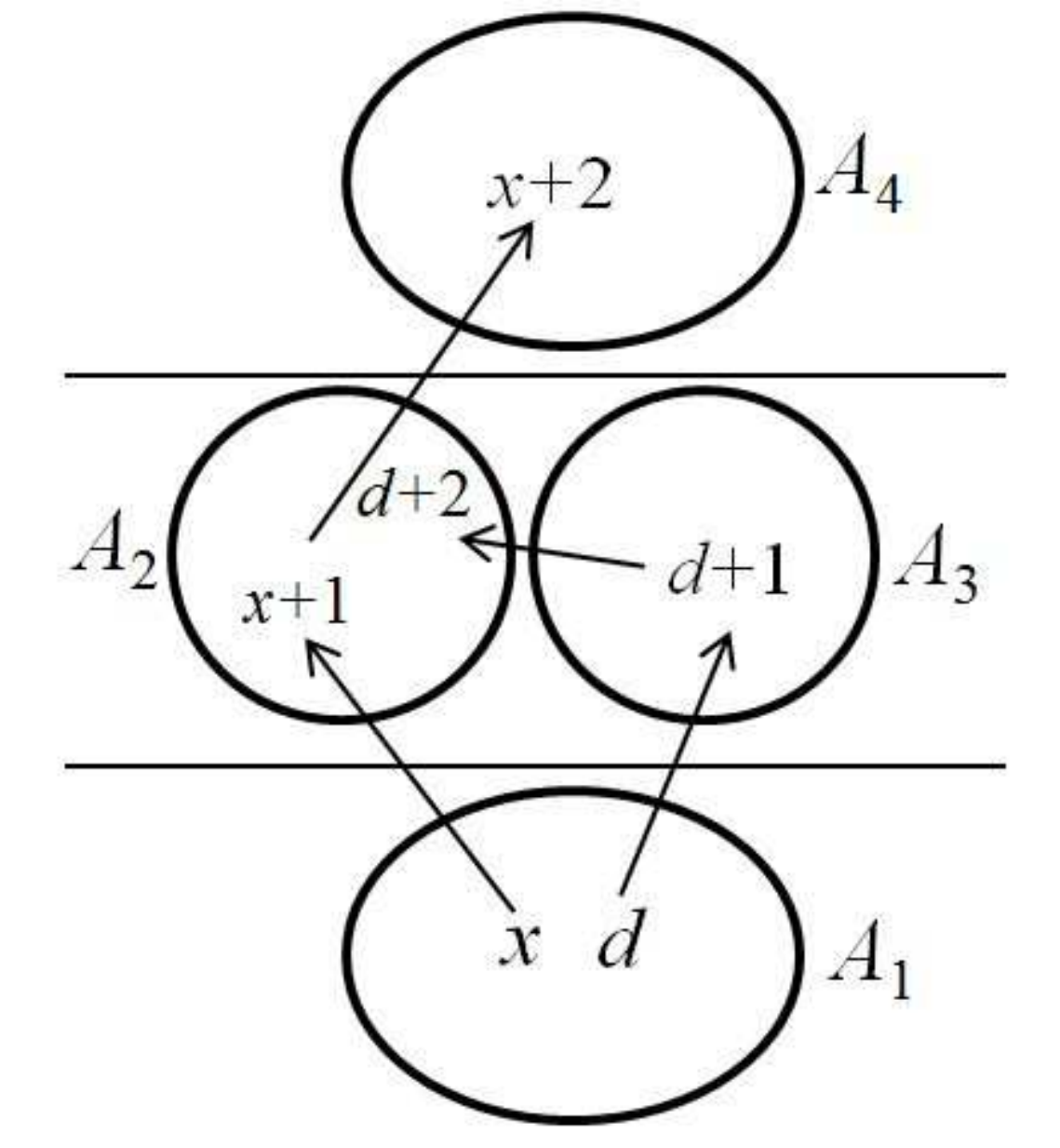}}
  \subfigure[]{\label{fig5d}\includegraphics[scale=0.23]{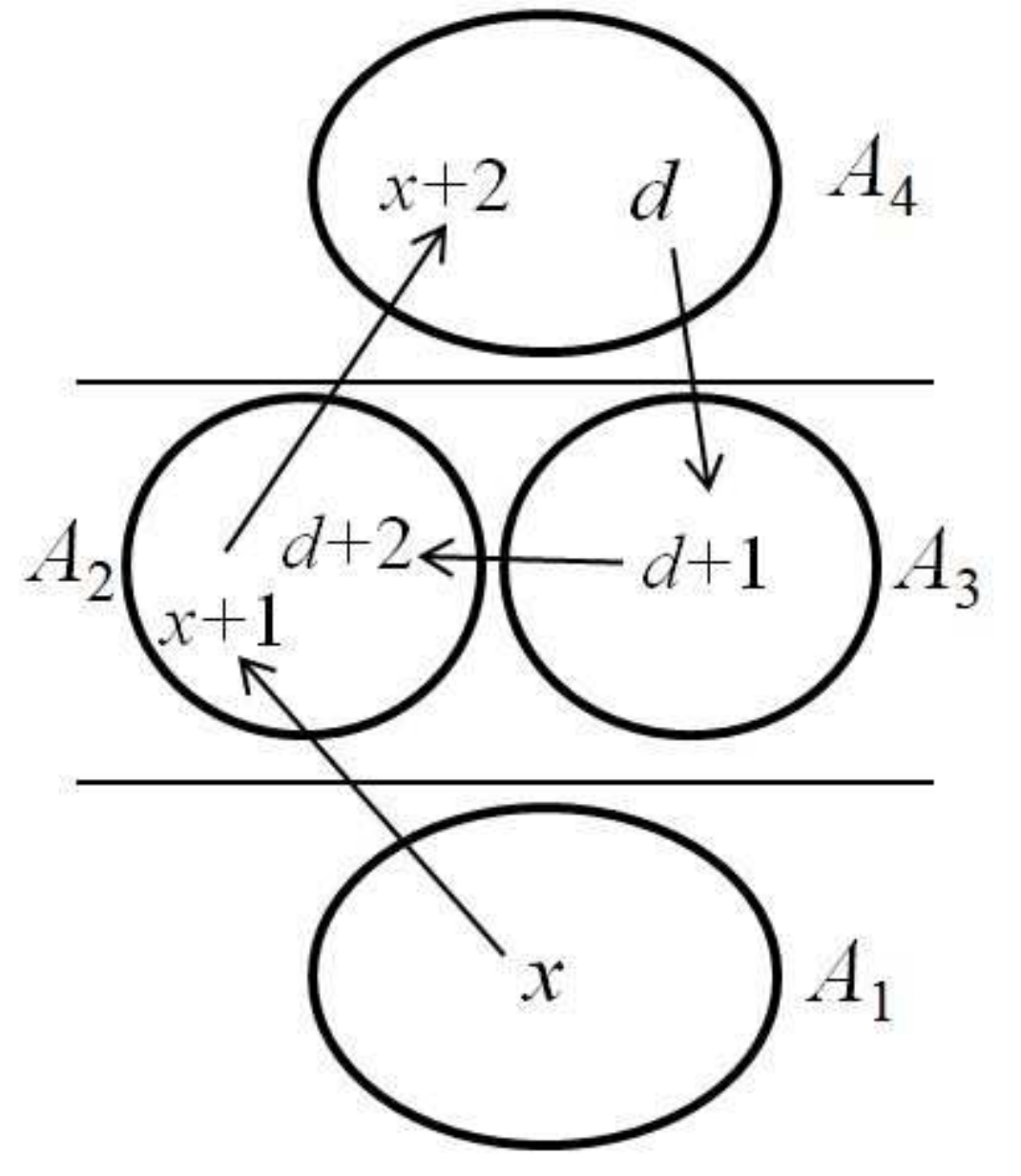}}
  \subfigure[]{\label{fig5e}\includegraphics[scale=0.23]{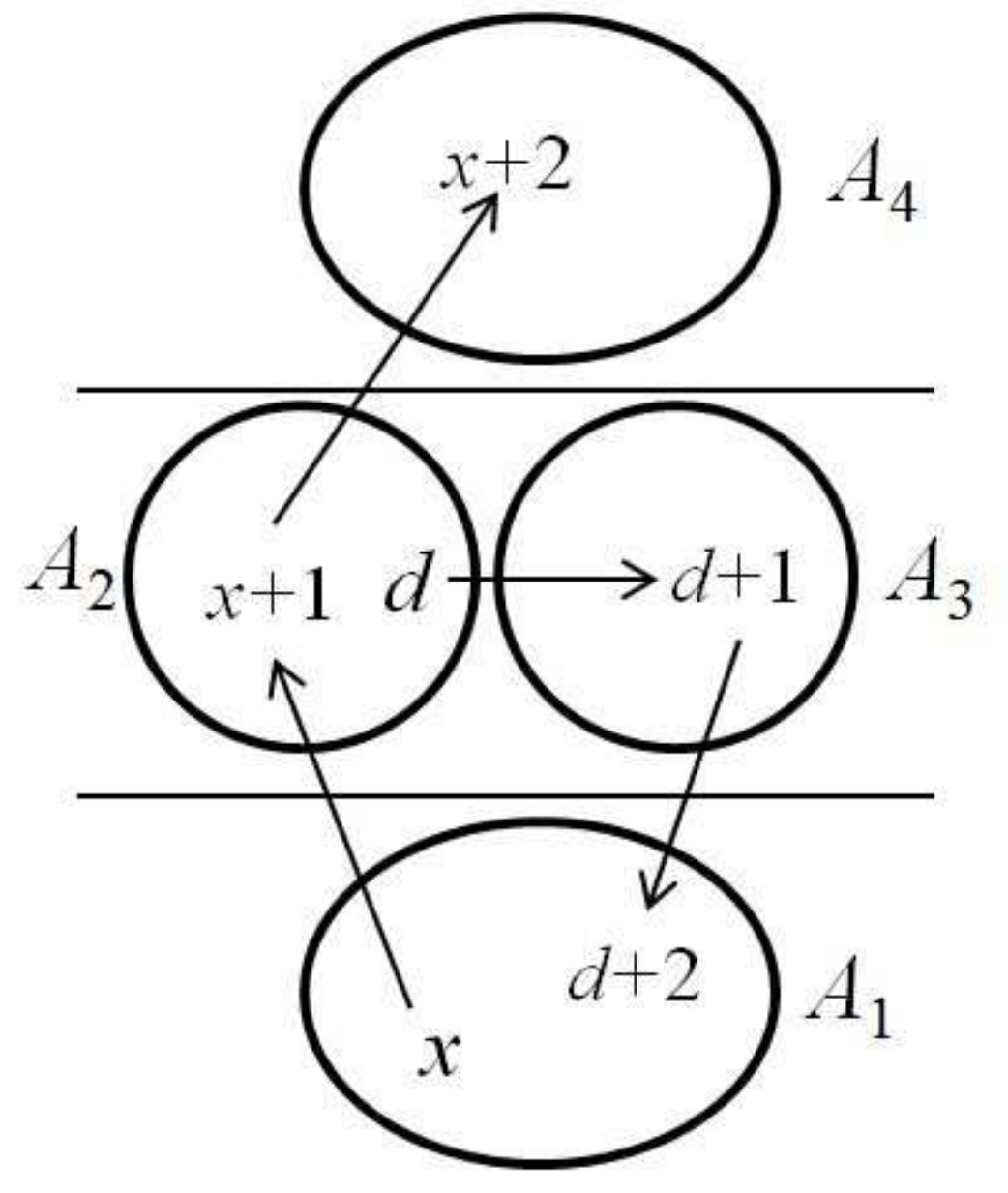}}
  \subfigure[]{\label{fig5f}\includegraphics[scale=0.23]{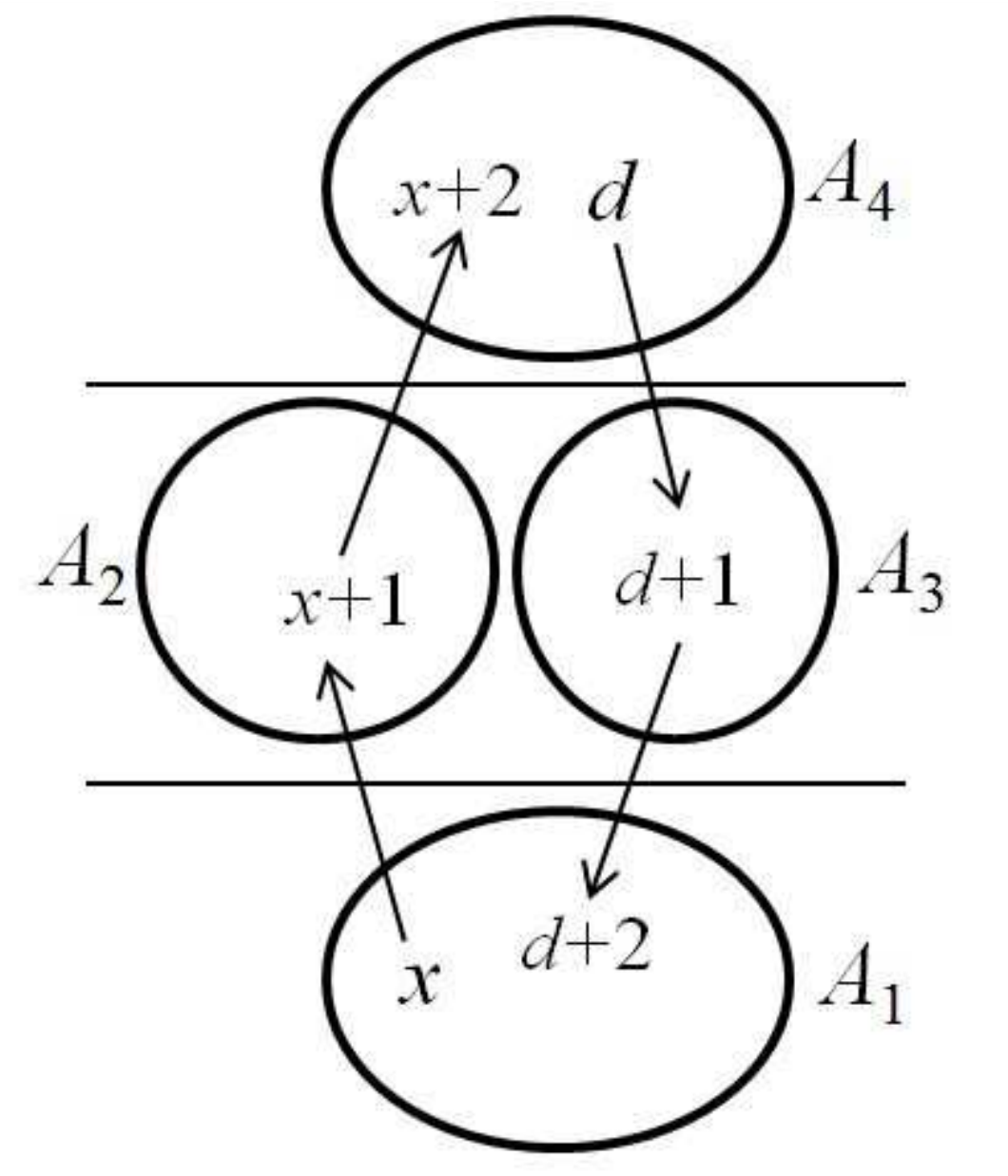}}
  \caption{}
  \label{fig5}
\end{figure}

\begin{proof}
 Assume the configuration in Figure~\ref{fig5a} exists. Since there is symmetry between the roles of $x$ and $d$, we may assume that $x>d$. Assume there exists $c\ne n$ in $A_2-(x+1)$ such that $c+1\not\in A_2$. We have $c+1\not\in A_1-x$ (Lemma~\ref{lemma:6}(iii)), $c+1\not\in A_3$ (Lemma~\ref{lemma:3}(b)) and $c+1\not\in A_4$ (Lemma~\ref{lemma:3}(a)).

Thus, we must have $c+1=x$. Let $k\ge0$ be minimal such that $c-k\in A_2$ but $c-k-1\not\in A_2$ (assuming there exists such $k$) and denote $u=c-k$. We have, $u-1\not\in A_1$, by Lemma~\ref{lemma:6}(ii), and $u-1\not\in A_4$, by Lemma~\ref{lemma:6}(v). If $u-1\in A_3$ it contradicts Lemma~\ref{lemma:3}(b). We conclude that such $k$ does not exist.

 Thus, if $A_2$ contains any element smaller than $x$, it must contain $\{1,2,\ldots,x-1\}$. This is impossible, since $d<x$. Hence, such $c$ does not exist and we must have that $A_2=\{x+1, n-l,\dots,n\}$ for some $l\ge 0$. Let $u=n-l$. As in the previous paragraph, $u-1\not\in A_1\cup (A_4-(x+2))\cup A_3$. So, we must have $u-1=x+2$. That is, $A_2=\{x+1, x+3,\dots,n\}$.

 Using a similar argument as for $A_2$ above we conclude that $A_3=\{1,2,\ldots,d-1,d+1\}$. Thus, $A_1\cup A_4=\{d,d+2,\ldots,x,x+2\}$. Since $d+2\in A_4$ and $x\in A_1$, there must exist $a\in A_4$ such that $a+1\in A_1$. This would contradict Lemma~\ref{lemma:6}(vi), unless $a=d+2$ and $a+1=x$. In this case we must have $A_1=\{x,d\}$ and $A_4=\{x+2,d+2\}$. Hence, $S(A_1)=x+d$ and $S(A_4)=x+d+4$. This yields a contradiction since $S(A_1)=S(A_4)-2$. We conclude that a configuration as in Figure~\ref{fig5a} is not possible.

 Now, if a configuration as in Figure~\ref{fig5b} exists, applying $\chi_{x,x+1}$ yields an equivalent partition with a configuration similar to the one in Figure~\ref{fig5a}. If a configuration as in Figure~\ref{fig5c} exists, applying $\chi_{x+1,x+2}$ yields a configuration similar to the one in Figure~\ref{fig5a}. If a configuration as in Figure~\ref{fig5d} exists, applying $\chi_{x,x+2}$ yields an equivalent partition with a configuration similar to the one in Figure~\ref{fig5c}. If a configuration as in Figure~\ref{fig5e} exists, applying $\chi_{x,x+2}$ yields a configuration similar to the one in Figure~\ref{fig5b}. Finally, If a configuration as in Figure~\ref{fig5f} exists, applying $\chi_{x,x+2}$ yields a configuration similar to the one in Figure~\ref{fig5a}.
 \end{proof}

\section{Proof of Theorem~\ref{thm1}}

\begin{proof}[\unskip\nopunct]
We assume, for contradiction, that $d^{\ip}_{\textrm{min}}>0$ for the given $\{p_i\}_{i=1}^4$. Let $\A=\{A_1,A_2, A_3 ,A_4\}$ be a minimal partition of $[n]$ implementing $\{p_i\}_{i=1}^4$, such that $\omega(\A)$ is minimal among all such minimal partitions. By Lemma~\ref{lemma:4}, we may assume that $S(A_1)=s^{n,4}-1$, $S(A_2)=S(A_3)=s^{n,3}$ and $S(A_4)=s^{n,4}+1$. Let $x\in A_1$ and $y\in A_4$ be such that $y-x=\omega(\A)$. We may assume that $x+1\in A_2$. By Lemma~\ref{lemma:6}(i), there are only two cases to consider: $y-x=2$ and $y-x=3$.
\medskip

First assume $y-x=2$. That is, $x\in A_1, x+1\in A_2$ and $x+2\in A_4$. It will be convenient to notice that we have the same setup as in Figure~\ref{fig4b}, with $X=A_1$, $A=B=A_2$ and $Y=A_4$. Let $B_1=A_1-x$, $B_2=A_2-(x+1)$, $B_3=A_3$ and $B_4=A_4-(x+2)$. We may make two assumptions:

\textbf{Assumption 1:} $n\not \in A_1$

\textbf{Assumption 2:} $\max_{A_2}<\max_{A_4}$.

(If $n\in A_1$ we apply $\chi_{x,x+2}$, which ensures both assumptions. If $n\not\in A_1$ and Assumption 2 does not hold, we apply $\chi_{x+1,x+2}$.)

Let $m_i$ be the maximal elements in $B_i$ for $i=1,\ldots,4$. By our assumptions, $m_1, m_2<n$. By Lemma~\ref{lemma:3}(a), $m_2+1\not\in A_4$ and by Lemma~\ref{lemma:6}(iii), $m_2+1\not\in B_1$. Thus, either $m_2+1\in A_3$ or $m_2+1=x$. Suppose $m_2+1\in A_3$. By Assumption 2, there exists $k\ge1$ such that $m_2+k\in A_3$ and $m_2+k+1\not\in A_3$ (otherwise $m_2+1,m_2+2,\ldots, n$ are all in $A_3$ and $m_2$ would be larger than any element of $A_4$). Suppose $m_2+k+1\in A_4$. If $k=1$ it contradicts Lemma~\ref{lemma:7}(b), and if $k>1$ it contradicts Lemma~\ref{lemma:3}(b). Now, suppose $m_2+k+1\in A_1$ and $m_2+k+1\ne x$. If $k=1$ it contradicts Lemma~\ref{lemma:7}(e), and if $k>1$ we contradict Lemma~\ref{lemma:3}(c). Thus, $m_2+k+1=x$ for some $k\ge 0$. In any case we have $x-1\in A_2\cup A_3$. In particular, $x-1\not\in A_1$.

Now, by Lemma~\ref{lemma:3}(a), $m_1+1\not\in A_2$, and by Observation~\ref{obs:2}, $m_1+1\not\in A_4$. It follows that $m_1+1\in A_3$.
Suppose there exists $l\ge1$ such that $m_1+l\in A_3$ and $m_1+l+1\not\in A_3$. Note that $m_1+l+1\ne x$, since we already know that $m_2+k+1=x$. Since $m_1$ is maximal in $A_1-x$, we must have $m_1+l+1\in A_2\cup A_4$. Suppose $m_1+l+1\in A_4$. If $l=1$, then $m_1+2\in A_4$ and we contradict Lemma~\ref{lemma:7}(a). If $l>1$, then $m_1+1\ne m_1+l$ and we contradict Lemma~\ref{lemma:3}(a).
Now, suppose $m_1+l+1\in A_2$. If $l=1$ it contradicts Lemma~\ref{lemma:7}(c). If $l>1$, it contradicts Lemma~\ref{lemma:3}(b). We conclude that $m_1+1,\ldots, n\in A_3$ and thus, $m_4<n$.

We have $m_4+1\not\in A_2$, by Assumption 2, and $m_4+1\not\in A_1$, by Lemma~\ref{lemma:6}(vi) and the fact that $x-1\in A_2\cup A_3$. So, we must have $m_4+1\in A_3$ and there exists $t\ge1$ such that $m_4+t\in A_3$ and $m_4+t+1\not\in A_3$ (since $m_1+1,\ldots, n\in A_3$). Clearly, $m_4+t+1\ne x+2$, so $m_4+t+1$ can only be in $A_1$, by Assumption 2. Now, $m_4+t+1\ne x$, since $m_2+k+1=x$. So, $m_4+t+1\in B_1$. If $t=1$ it contradicts Lemma~\ref{lemma:7}(f), and if $t>1$, we have a contradiction to Lemma~\ref{lemma:3}(e). This concludes the case where $y-x=2$.

\medskip

Now assume $y-x=3$. By Lemma~\ref{lemma:6}, we have $x+1\in A_2, x+2\in A_3$ and $y=x+3\in A_4$. It will be convenient to notice that we have a setup similar to the one in Figure~\ref{fig4a}, with $X=A_1$, $A=A_2$, $B=A_3$, $Y=A_4$ and $y=x+3$. Let $B_1=A_1- x, B_2=A_2- (x+1), B_3=A_3- (x+2)$ and $B_4=A_4-(x+3)$.

Let $z\in B_1$ and assume $z\ne n$. We know that $z+1\not\in A_4$ by Observation~\ref{obs:2}, $z+1\not\in A_2$ by Lemma~\ref{lemma:6}(ii),  and $z+1\not\in A_3$ by Lemma~\ref{lemma:3}(a). Thus, $z+1\in A_1$.
Now, let $z\in B_2$ and $z\ne n$. We have, $z+1\not\in A_4$ by Lemma~\ref{lemma:3}(a), and $z+1\not\in B_1$ by Lemma~\ref{lemma:6}(iii). Also, $z+1\not\in A_3$, by Lemma~\ref{lemma:3}(b). Thus, $z+1\in B_2$ or $z+1=x$. It follows that either $B_1$ or $B_2$ is equal to $\{t,t+1,\ldots,x-1\}$ for some $t\ge 1$ and the other is equal to $\{s,s+1,\ldots, n\}$ for some $s > x+3$. We may assume $A_1=\{x,s,s+1,\ldots,n\}$ and $A_2=\{t,t+1,\ldots,x-1,x+1\}$ (by applying $\chi_{x,x+1}$ if necessary).

Let $z>1$ be an element of $B_4$. We have $z-1\not\in A_1$ by Observation~\ref{obs:2}, $z-1\not\in A_2$ by Lemma~\ref{lemma:3}(a), and $z-1\not\in A_3$ by Lemma~\ref{lemma:6}(iv). Thus, $z-1\in A_4$. Let $z>1$ be an element of $B_3$. We have $z-1\not\in A_1$ by Lemma~\ref{lemma:3}(a), $z-1\not\in B_4$ by Lemma~\ref{lemma:6}(v), and $z-1\not\in A_2$, by Lemma~\ref{lemma:3}(b). Thus, $z-1\in B_3$ or $z-1=x+3$.
We conclude that either $B_3$ or $B_4$ is equal to $\{1,2,\ldots,t-1\}$ and the other is equal to $\{x+4,\ldots, s-1\}$ for the same $t$ and $s$ as above. Since none of the $A_i$'s has size 1, we must have that $1<t\le x-1$ and $x+3<s\le n$. By applying $\chi_{x+2,x+3}$ if necessary, we may assume that $A_3=\{1,2,\ldots,t-1,x+2\}$  and  $A_4=\{x+3,x+4,\ldots,s-1\}$ (Figure~\ref{fig6a}).

 \begin{figure}[h!]
  \centering
  \subfigure[]{\label{fig6a}\includegraphics[scale=0.3]{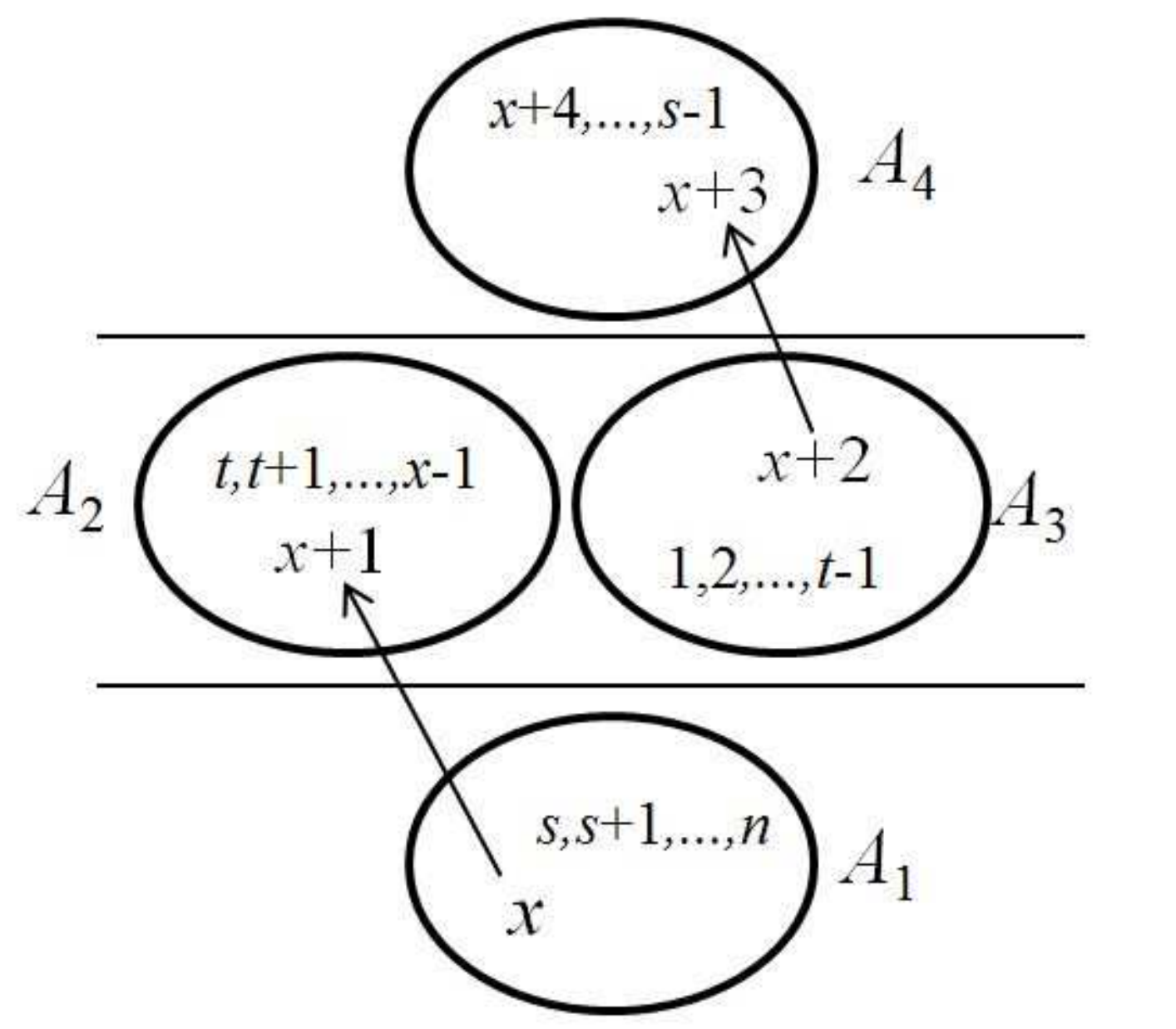}}
  \subfigure[]{\label{fig6b}\includegraphics[scale=0.3]{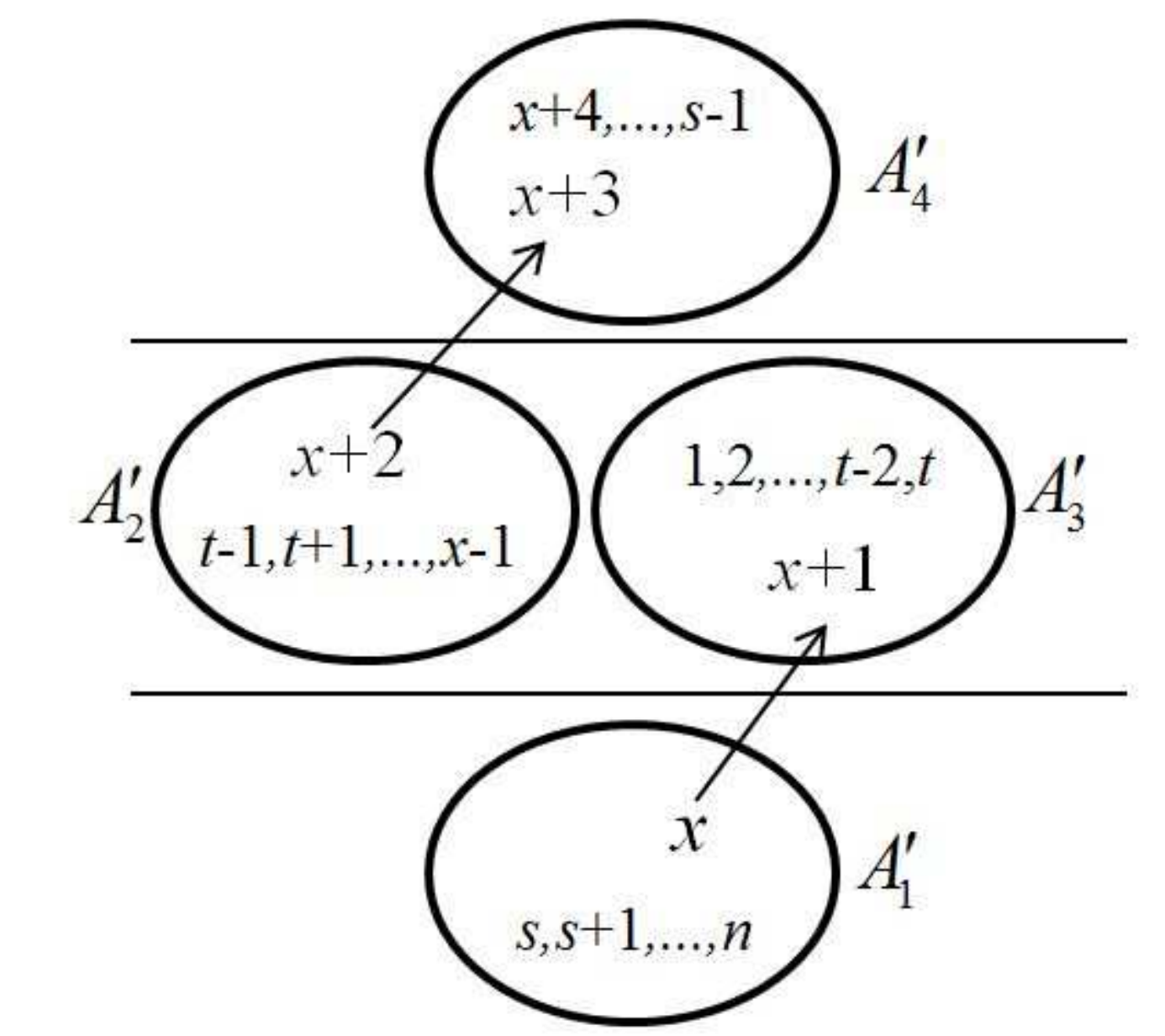}}
  \caption{}
  \label{fig6}
\end{figure}

Note that $s^{n,4}=S(A_4)-1\ge (x+3)+(x+4)-1=2x+6$. Thus, $|A_2|>2$, that is, $t<x-1$. Let $\A'$ be partition consisting of $A_1'=A_1, A_2'=A_2\setminus\{t,x+1\}\cup\{t-1,x+2\},  A_3'=A_3\setminus\{t-1,x+2\}\cup\{t,x+1\}$, and $A_4'=A_4$ (Figure~\ref{fig6b}). Note that $S(A_i')=S(A_i)$ for $i=1,\ldots,4$, and thus $\A'$ is also minimal with the same minimal width. Since $t<x-1$, we have $t+1\in A_2'$ and we have a contradiction to Lemma~\ref{lemma:3}(b). This completes the proof.

\end{proof}

\begin{remark}
In the case $k=2$ there is yet another simple proof:

Suppose $p_1+p_2=n$, $p_1\le p_2$ and $\sum_{i=1}^{p_1}(n-i+1)\ge s^{n,2}$ (Condition~(\ref{nec_cond})). Consider the following two partitions of $[n]$: $\A=\{A_1,A_2\}$ and $\B=\{B_1,B_2\}$, such that $|A_i|=|B_i|=p_i$, for $i=1,2$, $A_1=\{1,\ldots,p_1\}$ and $B_1=\{n-p_1+1,\ldots,n\}$. We have $|A_1|<s^{n,2}$ since $p_1\le p_2$ and $|B_1|\ge s^{n,2}$ by (\ref{nec_cond}). We show that we can switch from $\A$ to $\B$ by a sequence of operations of the form $\chi_{a,a+1}$. Thus, at some point along the way we must have an equitable partition.

We start with partition $\A$ and apply $\chi_{1,2}\circ\chi_{2,3}\circ\cdots\circ\chi_{p_1-1,p_1}\circ\chi_{p_1,p_1+1}$. This results in the partition consisting of $\{2,3,\ldots,p_1,p_1+1\}$ and $\{1,p_1+2,\ldots,n\}$. Then, we apply $\chi_{2,3}\circ\cdots\circ\chi_{p_1+1,p_1+2}$, resulting in $\{3,4,\ldots,p_1+2\}$ and $\{1,2,p_1+3,\ldots,n\}$, and so on. Eventually we arrive at partition $\B$.

It might be possible to generalize this continuity approach to higher $k$'s by applying a higher dimensional continuity technique, such as Sperner's theorem. The problem is to define the right division into $(k-1)$-dimensional simplices.
\end{remark}

\section*{Acknowledgments}
The author thanks the referees for a very thorough reading of the manuscript.


\end{document}